\documentclass[10pt]{article}

\usepackage{amsmath,amsfonts,amssymb,amsthm}
\usepackage{enumerate}
\usepackage[hyphens]{url}
\usepackage{tcolorbox}
\usepackage{hyperref}
\usepackage{xspace}
\usepackage[margin=3.5cm]{geometry}
\usepackage{marvosym}
\usepackage{thmtools}

\theoremstyle{plain}
\newtheorem{theorem}{Theorem}
\newtheorem{lemma}[theorem]{Lemma}
\newtheorem{corollary}[theorem]{Corollary}
\newtheorem{proposition}[theorem]{Proposition}

\theoremstyle{definition}

\newcommand\N{\mathbb{N}}

\newcommand\R{\mathbb{R}}

\newcommand{\dd}[2]{\frac{\mathrm{d} #1}{\mathrm{d} #2}}

\DeclareMathOperator{\im}{im} 
\DeclareMathOperator{\rank}{rank}
\DeclareMathOperator{\lin}{span}
\DeclareMathOperator{\aff}{aff}
\DeclareMathOperator{\codim}{codim} 
\DeclareMathOperator{\int_}{int}
\DeclareMathOperator{\relint}{relint}
\DeclareMathOperator{\crit}{Crit}

\DeclareMathOperator{\zero}{zero}

\DeclareMathOperator{\e}{e} 
\DeclareMathOperator{\diag}{diag}
\DeclareMathOperator{\id}{Id}

\newcommand{\ones}{1}
\newcommand{\zeros}{0}
\newcommand{\trans}{\mathsf{T}}

\newcommand{\cayley}{J_C}
\newcommand{\IF}{I}
\newcommand{\JF}{J}
\newcommand{\cons}{l}

\newcommand{\ff}{\Phi}
\newcommand{\vv}{\mathcal{V}}
\newcommand{\bb}{\mathcal{B}}

\definecolor{darkgreen}{rgb}{0.0,0.7,0.0}

\newcommand\gray[1]{{\textcolor{gray}{#1}}}

\newcommand\blfootnote[1]{%
  \begin{NoHyper}
  \renewcommand\thefootnote{}\footnote{#1}%
  \addtocounter{footnote}{-1}%
  \end{NoHyper}
}

\usepackage{algorithm, algpseudocode}

\makeatletter

\begin{document}

\title{Generic finiteness of solutions to parametrized systems \\ of generalized polynomial inequalities}

\author{
Stefan M\"uller\textsuperscript{\Letter}
}

\maketitle


\begin{abstract}
We study parametrized systems of {\em generalized} polynomial inequalities, 
$(c \circ x^B) \in C$,
for a positive variable vector $x \in \R^n_>$, 
a {\em real} exponent matrix $B \in \R^{n \times m}$, 
a positive parameter vector $c \in \R^m_>$, 
and a positive, blunt coefficient cone $C \subset \R^m_>$.
This setting includes (i)~generalized polynomial equations, $A \, (c \circ x^B) = 0$,
arising e.g.\ from power-law dynamical systems with real exponents
(and involving a real coefficient matrix $A \in \R^{n' \times m}$), 
and (ii)~generalized semi-algebraic sets.

Our main result guarantees the generic finiteness of solutions modulo unavoidable exponential fibers.
Building on a framework for parametrized {\em polynomial} systems developed by M\"uller and Regensburger, 
we reformulate the problem as a system of {\em binomial} equations on the coefficient set $P = C \cap \Delta \subseteq \R^m_>$,
where $\Delta$ is a product of (standard) simplices.
Solvability is governed by the monomial-dependency subspace $D$.

Using elementary arguments from analysis and topology,
including a logarithmic boundary analysis and Sard's theorem,
we control the accumulation of solutions on the boundary and in the interior of $P$.
We show that, if $\dim P \le \dim D$,
then generically in both the measure-theoretic and topological sense,
there are finitely many solutions on $P$
(and hence finitely many exponential fibers of solutions in $\R^n_>$).

As an application, we prove generic finiteness of steady states (within forward-invariant sets)
for parametrized power-law dynamical systems with real exponents.


\vspace{2ex}

\vspace{2ex}
\noindent
{\bf MSC classes.} 
14P05,
14P10,
58C25,
57R45,
37N25,
92C42
\end{abstract}

\blfootnote{
\scriptsize

\noindent
{\bf Stefan~M\"uller} \\
Faculty of Mathematics, University of Vienna, Oskar-Morgenstern-Platz 1, 1090 Wien, Austria \\
\Letter \, st.mueller@univie.ac.at 
}



\section{Introduction}

Parametrized power-law dynamical systems (with real, not necessarily integer exponents)
are used to model a variety of chemical, biological (e.g.\ ecological or epidemiological), economic, and engineering processes.
For example, they arise from (chemical) reaction networks with (generalized) mass-action kinetics;
see e.g.~\cite{Horn1972,HornJackson1972,Feinberg1972} and \cite{MuellerRegensburger2012,MuellerRegensburger2014}.
The analysis of their steady states leads to parametrized systems of generalized polynomial equations and inequalities,
given by finite sums of monomial terms 
(where a monomial term is a monomial times a positive parameter).

Solution sets of (generalized) polynomial systems may consist of infinite exponential fibers.
Hence, generic finiteness questions are both natural and subtle:
one seeks conditions ensuring that, for generic parameter values, only finitely many solution fibers occur.
By ``generic'' we mean that finiteness holds almost everywhere and is robust;
that is, the ``good'' parameter set is open and dense, and the ``bad'' set has Lebesgue measure zero.

In this work,
we consider parametrized systems of generalized polynomial inequalities~\cite{MuellerRegensburger2023a} of the form
\[
(c\circ x^B)\in C,
\]
where $x \in \R^n_>$ is a positive variable vector, $c \in \R^m_>$ is a positive parameter vector, 
$B \in \R^{n \times m}$ is a real exponent matrix, and $C \subseteq \R^m_>$ is a positive, blunt coefficient cone. 
This formulation includes (i) generalized polynomial equations, $A \, (c \circ x^B) = 0$, where $A \in \R^{n' \times m}$ is a real coefficient matrix,
and (ii) ``parametrized generalized semi-algebraic sets'', if $C$ is given by linear equations and inequalities;
see Section~\ref{sec:semi} for a formal definition.
In particular, this formulation includes steady-state problems for parametrized power-law dynamical systems
with {\em real} exponents.

The main purpose of this work is twofold. 
First, we extend generic finiteness results for parametrized polynomial equations to {\em generalized} polynomial equations (with real exponents) 
and to polynomial {\em inequalities}. 
Second, we give a proof based only on elementary tools from linear algebra, polyhedral geometry, analysis, and topology.
In particular, our argument avoids the use of algebraic geometry 
and is therefore applicable without the assumption of integer exponents.

Building on the framework developed by M\"uller and Regensburger~\cite{MuellerRegensburger2023a},
we pass from the original variable space $\R^n_>$ to a coefficient
set
\[
P=C\cap\Delta \subseteq \R^m_> ,
\]
where $\Delta$ is a product of (standard) simplices determined by a decomposition of the coefficient cone $C$. 
The relevant exponent data are encoded by the monomial-dependency subspace
\[
D=\ker B\cap\ker\cayley ,
\]
where $\cayley$ is a corresponding Cayley matrix.
Solution sets on $P$ are fibers of a logarithmic-linear map from $P$ to a space of dimension $\dim D$.
The principal difficulty is that the logarithm may diverge along the boundary of the positive orthant. 
We control this phenomenon by an elementary logarithmic-polytope argument, 
which identifies a lower-dimensional set of boundary parameter values. 
Properness away from these values, combined with Sard's theorem~\cite{lee2013,guilleminpollack1974}, guarantees generic finiteness under the criterion
\[
\dim P\le\dim D.
\]
This dimension condition also has the expected interpretation. 
If $\dim P<\dim D$, then solution sets are generically empty. 
If $\dim P=\dim D$, then solution sets on $P$ are generically finite (and solution sets on $\R^n_>$ have finitely many exponential fibers). 
Thus our main result simultaneously detects generic nonexistence and finiteness of solutions.
Our result also applies to restricted parameter sets with variable and fixed parameters,
provided that the corresponding logarithmic-linear lumping map has full rank.

For generalized polynomial {\em inequalities}, we consider two complementary equation systems. 
The first retains all monomials and replaces the coefficient cone by its linear-span relaxation.
The second discards phantom monomials (not appearing in the equation part).
The resulting sufficient generic-finiteness criteria need not imply one another,
and both can be checked directly from the coefficient cone and exponent data.

For classical (Laurent) polynomial systems, 
generic finiteness and root counts are studied by Bernstein--Kushnirenko theory, 
which relates the number of isolated solutions in the complex torus to the Newton polytopes of the equations~\cite{bernshtein1975,kushnirenko1976}.
In that framework, genericity is assumed with respect to the set of complex coefficients.
Our setting differs in that we fix real coefficients~$A$, but allow for positive parameters~$c$, real exponents~$B$, and inequalities. 
Fewnomial theory also concerns systems of generalized polynomial equations (with real exponents)~\cite{khovanskii1991,bihan2007gale}. 
It provides combinatorial upper bounds for the number of nondegenerate, isolated positive solutions 
in terms of the number of monomials (and variables), 
under the explicit assumption that the solution set is finite.
Finally,
in reaction network theory, injectivity criteria guarantee uniqueness of solutions.
They apply under structural sign or graph-theoretic assumptions on the network~\cite{craciunfeinberg2005i, banajicraciun2010, feliuwiuf2013, Mueller2016, BanajiPantea2016, FeliuMuellerRegensburger2019}.

The present work is most closely related to results of Feliu and collaborators on nondegenerate solutions of parametrized ``power-law'' systems (with integer exponents),
arising from reaction networks with mass-action kinetics.
For a recent generic finiteness result, see~\cite{banaji2026positive} 
(which builds on \cite{MuellerRegensburger2023a} and \cite{feliu2025generic,feliu2026generic}).
In the setting of tropical geometry, systems in which a monomial term can appear in several equations
have been called ``vertically parametrized''.
Their approach utilizes algebraic-geometric properties of polynomial systems to analyze these vertically parametrized structures.
Our approach is complementary.
We work directly with real exponents and replace algebraic-geometric notions and methods (ring, ideal, variety, Zariski topology, or Gr\"obner bases)
by linear-algebraic, polyhedral, and differential-topological arguments.
In particular, we build on the underlying geometry: the coefficient set (a polytope, in the equation case) 
and the monomial-dependency subspace.

When the coefficient cone is definable,
an alternative route to generic finiteness is available through o-minimal geometry. 
On the positive orthant, generalized monomials (with real exponents) are definable in the o-minimal structure $\R_{\exp}$~\cite{wilkie1996},
and the solution sets considered here belong to definable families~\cite{vandendriesmiller1996}.
General finiteness and stratification results for definable maps can therefore yield related genericity results~\cite{vandendries1998,coste2000}. 
We instead give an elementary proof tailored to our parametrized systems of generalized polynomial equations $A \, (c \circ x^B) = 0$
and inequalities $(c \circ x^B) \in C$. 
Besides avoiding substantially heavier machinery, 
this approach is more directly accessible in the reaction-network and power-law-systems community 
and makes the relevant boundary analysis and dimension condition transparent.


\subsubsection*{Organization of the work}

In Section~\ref{sec:equ},
we first recall the framework for parametrized systems of generalized polynomial {\em equations} developed in~\cite{MuellerRegensburger2023a}
and define the relevant logarithmic-linear map.
In the technical Sections~\ref{sec:boundary} and \ref{sec:critical},
we study boundary and critical parameters, respectively.
In Section~\ref{sec:main}, we establish the dimension condition $\dim P \le \dim D$ as a generic-finiteness criterion,
and in Section~\ref{sec:main_original}, we allow for restricted parameter sets with variable and fixed parameters.
Finally, in Section~\ref{sec:app}, we apply our results to steady states of parametrized power-law dynamical systems with real exponents.
For completeness,
we deal with the case $\dim P > \dim D$ in Section~\ref{sec:under}.

In Section~\ref{sec:inequ}, we treat parametrized systems of generalized polynomial {\em inequalities}
as equation systems by linear-span relaxation (Section~\ref{sec:relaxation}) and dimension reduction (Section~\ref{sec:reduction}).
For the convenience of the reader,
we explicitly define parametrized generalized semi-algebraic sets in Section~\ref{sec:semi},
and exhibit that they are covered by the framework developed in~\cite{MuellerRegensburger2023a}.


\subsubsection*{Notation}

We denote the positive (nonnegative) real numbers by $\R_>$ ($\R_\ge$).
Throughout the work, we use index notation: for a finite index set $I$,
we write $\R^I$ for the real vector space of vectors $x=(x_i)_{i \in I}$ with $x_i \in \R$,
and analogously we write $\R^I_>$ ($\R^I_\ge$).
For $I=\{1,\ldots,n\}$, we have the standard case $\R^I=\R^n$.
We write $x>0$ for $x \in \R^I_>$ and $x \ge 0$ for $x \in \R^I_\ge$.

For vectors $x,y \in \R^I$, we denote their componentwise (Hadamard) product by $x \circ y \in \R^I$.
For $x \in \R^I_>, \, y \in \R^I$, we define the (generalized) monomial $x^y = \prod_{i \in I} (x_i)^{y_i} \in \R_>$,
and for $x \in \R^I_>, \, Y \in \R^{I \times J}$, 
we define the vector of monomials $x^Y \in \R^J_>$ via $(x^Y)_j =x^{y(j)}$,
where $y(j)$ is the column of $Y$ with index $j \in J$.


\section{Parametrized generalized polynomial equations} \label{sec:equ}

Let $A \in \R^{n' \times m}$ be a real coefficient matrix, $B \in \R^{n \times m}$ be a {\em real} exponent matrix, 
and $c \in \R^m_>$ a positive parameter vector.
We consider the parametrized system of generalized polynomial equations
\begin{equation}
A \left( c \circ x^B \right) = \zeros 
\end{equation}
for the positive variable vector $x \in \R^n_>$.

The problem can be written in the general setting of inequalities as $(c \circ x^B) \in C$
with (positive, blunt) coefficient cone ${C = \ker A \cap \R^m_>}$,
a ``subspace cone'' in the terminology of~\cite{MuellerRegensburger2016}.
As a necessary condition for the existence of solutions, we assume $C \neq \emptyset$.
We further assume that $C = C^1 \times \cdots \times C^\ell$,
that is, the coefficient cone has a direct product decomposition, inducing $\ell$ classes of (monomial) indices.
We can always choose the trivial decomposition with $\ell=1$.
To maximally reduce the problem dimension, we can choose the finest decomposition.
The coefficient set $P = C \cap \Delta$ arises from the intersection with a product $\Delta = \Delta^{m_1-1} \times \cdots \times \Delta^{m_\ell-1}$ of (standard) simplices with $m_1+\ldots+m_\ell=m$.
It is a relatively open polytope.
We write $\overline{P}$ for the closure of $P$ (in $\R^m_\ge$) and $\partial P = \overline{P} \setminus \relint P = \overline{P} \setminus P$ for the relative boundary.

Besides the coefficient set (arising from the coefficient matrix $A$),
the monomial-dependency subspace (arising from the exponent matrix $B$) is the most crucial geometric object.
It is given by $D = \ker B \cap \ker \cayley$, where $\cayley \in \{0,1\}^{\ell \times m}$ is a Cayley matrix (having $\ell \times \ell$ blocks, in particular, rows of ones in the diagonal blocks and zeros elsewhere) indicating classes of indices.

\paragraph{Main previous result.}
By \cite[Theorem~5]{MuellerRegensburger2023a}, the solution set 
\begin{equation}
X_c = \{ x \in\R^n_> \mid A \left( c \circ x^B \right) = \zeros\}
\end{equation}
is given by
\[
X_c = (Y_c / c)^E \circ \e^{L^\perp} ,
\]
where
\begin{equation}
Y_c =  \{ y \in P \mid y^H = c^H \}
\end{equation}
is the solution set on the coefficient set, involving a basis matrix $H \in \R^{m \times d}$ of the monomial-dependency subspace $D \subset \R^m$,
that is, $D = \im H$ and $d = \dim D$.
Note that solvability depends on only $d<m$ lumped parameters $c^H$.

Finally, $E \in \R^{n \times m}$ is an exponentiation matrix (explicitly mapping elements of~$Y_c$ to elements of~$X_c$), see~\cite{MuellerRegensburger2023a},
and $L = B \, (\ker \cayley) \subseteq \R^n$ is the monomial difference subspace.
Both $E$ and $L$ do not depend on~$c$.
We note that 
\[
d = m - \ell - \dim L .
\]

If $L \neq \R^n$, then a nonempty solution set~$X_c$ necessarily contains exponential fibers of the form $x \circ \e^{L^\perp}$ and cannot be finite.
Hence, we consider the generic finiteness of the solution set~$Y_c$ (on the coefficient set), 
that is, the generic finiteness of exponential fibers in the (original) solution set~$X_c$.

Note that the dimension of $L$ depends on the decomposition of $C$ we choose. 
A finer decomposition leads to smaller $\ker \cayley$ and hence may lead to smaller $L = B \, (\ker \cayley)$ and hence larger $L^\perp$. 
That is, the maximal fibers may be immediately visible only for the finest decomposition.
In the trivial decomposition (or other decompositions between trivial and finest), they may be hidden in $Y_c$ which then cannot be finite.

\paragraph{Dimensions.} 
As we will show, generic finiteness (of exponential fibers) holds if the number of variables is less than/equal to the number of equations plus the dimension of the fibers, that is,
\[
n \le \rank A + \dim L^\perp ,
\]
or, equivalently, 
\[
\dim L \le \rank A .
\]
Since $\dim C = \dim (\ker A) = m - \rank A$ and $\dim P = \dim C - \ell$, this is further equivalent to
\[
\dim P \le d .
\]


\paragraph{Logarithmic-linear map.}
We are interested in the generic finiteness of the solution set $Y_c$ for positive parameters $c \in \R^m_>$.
Instead of solving $y^H=c^H$ for $y \in P$, we consider $H^\trans \ln y = H^\trans \ln c =: \gamma$
for the logarithmic lumped parameter $\gamma \in \R^d$.
Correspondingly, we define the logarithmic-linear map
\begin{equation} \label{eq:map}
\ff \colon P \to \R^d, \quad y \mapsto H^\trans \ln y
\end{equation}
and
\begin{equation}
\mathcal{Y}_\gamma = \{ y \in P \mid \ff(y) = \gamma \} ,
\end{equation}
for $\gamma \in \R^d$.

\subsection{Boundary parameter values} \label{sec:boundary}

First, we are interested in the set of non-proper parameter values (that have solution sets approaching $\partial P$),
\[
V_\text{np} =
\left\{
\gamma\in\R^d \; \middle| \;
\exists (y^k)_{k \in \N} \subseteq P 
\text{ and } y^* \in \partial P
\text{ s.t. }
y^k \to y^*
\text{ and }
\ff(y^k)=\gamma, \, \forall k
\right\}.
\]
These are the parameter values $\gamma$ for which $\mathcal{Y}_\gamma$ is not relatively compact in $P$. 
Clearly, this set is contained in the set of boundary parameter values,
\begin{equation}
V_\mathrm{bd} =
\left\{
\gamma\in\R^d \; \middle| \;
\exists (y^k)_{k \in \N} \subseteq P 
\text{ and } y^* \in \partial P
\text{ s.t. }
y^k \to y^*
\text{ and }
\ff(y^k)\to\gamma
\right\} .
\end{equation}
That is, $V_\text{np} \subseteq V_\mathrm{bd}$,
and hence
\[
\dim V_\text{np} \le \dim V_\mathrm{bd} .
\]


\begin{lemma}[Boundary-value dimension] \label{lem:bv_dim}
The set of boundary parameter values has codimension at least one,
\[
\dim V_\mathrm{bd} \le \dim P - 1 ,
\]
using the convention $\dim \emptyset = - 1$.
\end{lemma}
\begin{proof}
Let
\[
Q=\overline P
\qquad\text{and}\qquad
T=H^\trans.
\]
Then,
\[
P=\relint(\overline P) = \relint Q,
\quad
\partial P = \overline P \setminus P = Q \setminus \relint Q ,
\]
and
\[
\ff(y) = H^\trans \ln y = T \ln y
\quad \text{for} \quad
y\in P = \relint Q .
\]
Hence,  the set of boundary parameter values agrees with the set of logarithmic
boundary values as defined in Lemma~\ref{lem:logpol},
\[
V_\mathrm{bd}
=
\partial_T(Q) ,
\]
and Lemma~\ref{lem:logpol}~(i) yields
\[
\dim V_\mathrm{bd}
=
\dim\partial_T(Q)
\le
\dim Q-1
=
\dim P-1.
\]
\end{proof}

\begin{lemma}[Logarithmic polytope] \label{lem:logpol}
Let $Q \subseteq \R^n_\ge$ with $Q \cap \R^n_> \neq \emptyset$ (and hence $\relint Q \subseteq \R^n_>$) be a polytope and $T \in \R^{d \times n}$.
Further, let
\[
\partial_T(Q)
=
\left\{
\lambda \in \R^d \mid
\exists (q^k)_{k \in \N} \subseteq \relint Q
\text{ and }
q^* \in Q \setminus \relint Q
\text{ s.t. }
q^k \to q^*
\text{ and }
T \ln q^k \to \lambda
\right\} 
\]
be the associated set of logarithmic boundary values. Then,
\begin{enumerate}[(i)]
\item
$\dim \partial_T(Q) \le \dim Q - 1$ and
\item
$\dim \overline{T \ln (\relint Q)} \le \dim Q$,
\end{enumerate}
using the convention $\dim \emptyset = - 1$.
\end{lemma}

\begin{proof}
We prove (i) and (ii) simultaneously by induction on $\dim Q$.

$\bullet$
\quad
$\dim Q = 0$:

In other words,
$Q = \{ q^* \}$ with $q^* \in \R^n_>$.
Hence,
$Q = \relint Q$ and $\overline{T \ln (\relint Q)}$ are singletons,
$Q \setminus \relint Q$ and $\partial_T(Q)$ are empty,
and both statements (i) and (ii) hold.

$\bullet$
\quad
$\dim Q \ge 1$:

$-$ 
\quad
statement (i): 

Let $G$ be a proper face of $Q$ and
\[
\partial_T(Q|G)
=
\left\{
\lambda \in \R^d \mid
\exists (q^k)_{k \in \N} \subseteq \relint Q
\text{ and }
q^* \in \relint G
\text{ s.t. }
q^k \to q^*
\text{ and }
T \ln q^k \to \lambda
\right\} .
\]
Since 
\[
\partial_T (Q) = \bigcup_{G < Q} \partial_T(Q|G) ,
\]
it remains to show 
\[
\dim \partial_T(Q|G) \le \dim Q - 1 .
\]

Let $\IF = \zero G = \{ i \mid q_i=0, \, \forall q \in G \}$ and $\JF = \{1,\ldots,n\} \setminus \IF$. 
For $q^* \in \relint G$,
\[
q^*_\IF = 0
\quad \text{and} \quad
q^*_\JF > 0 .
\]
Let $G_\JF = \pi_\JF (G) \subseteq \R^\JF_\ge$ and note $\dim G_\JF = \dim G$.

\underline{If} $\relint G \subseteq \R^n_>$, 
that is, $I = \emptyset$,
then
\[
\partial_T(Q|G)=T\ln(\relint G).
\]
Every point of $\relint G$ can be approached from $\relint Q$,
and $T\ln$ is continuous and locally Lipschitz on $\relint G$.
Exhausting $\relint G$ by compact subsets yields
\[
\dim \partial_T(Q|G)
=
\dim (T\ln(\relint G))
\le
\dim G
\le
\dim Q-1.
\]

\underline{Otherwise},
$\IF \neq \emptyset$.

Let $C_G \subseteq \R^n$ be the inward transverse cone from $G$ into $Q$
and note that $s_\IF \ge 0$ for $s \in C_G$.
In particular, $\dim C_G = \codim_Q G$.
Near $\relint G$, points in $\relint Q$ can be written as
\[
q=q^*+s
\quad \text{with } q^* \in \relint G \text{ and } s \in \relint C_G .
\]
In particular,
\[
q_\IF=s_\IF, \quad
q_\JF= q^*_\JF+s_\JF .
\]
Clearly, $s \to \zeros$ as $q \to q^*$.
Further, let 
\[
V_G = \pi_\IF(C_G) \cap \Delta_\IF \subseteq \R^I_\ge
\quad \text{with} \quad
\Delta_\IF = \left\{ q \in \R^I_\ge \mid \sum_{i \in I} q_i=1 \right\}
\]
be a projective section of $\pi_\IF (C_G)$.
In particular, $\dim \pi_\IF(C_G) \le \dim C_G$
and hence
\[
\dim V_G \le \dim C_G - 1 = \codim_Q G - 1 \le \dim Q - 1 .
\]
Further, since $\relint Q \subseteq \R^n_>$ and $C_G$ is the inward transverse cone from $G$ into $Q$, 
${\relint V_G \subseteq \R^I_>}$.
Points in $\relint \pi_\IF(C_G)$ can be written uniquely as
\[
s_\IF=rv
\]
with radial scale $r = \sum_{i \in I} s_i > 0$ and projective normal direction $v = s_\IF / r \in \relint V_G$.
That is,
\[
\ln s_\IF = \ones_\IF \cdot \ln r + \ln v .
\]
Clearly, $r \to 0$ as $s \to \zeros$.

With $T = \begin{pmatrix} T_\IF & T_\JF \end{pmatrix}$ and $b = T_\IF \ones_\IF$,
\begin{equation} \label{eq:limit}
\begin{aligned}
T \ln q &= T_\IF \ln s_{\IF} + T_\JF \ln(q^*_\JF+s_\JF) \\
&=
T_\IF \ones_\IF \ln r + T_\IF \ln v + T_\JF \ln(q^*_\JF+s_\JF) \\
&= b \, \ln r + T_\IF \ln v + T_\JF \ln(q^*_\JF+s_\JF) ,
\end{aligned}
\end{equation}
where we omit the superscripts ($\phantom{}^k$) for the sequences $q,s,r,v$.

As $q \to q^*$, 
the third term converges to $T_\JF \ln q^*_\JF$.
To analyze the second term, we perform a case distinction regarding the first term.
\begin{itemize}
\item[$\bullet$] 
$b=0$:

Using Eqn.~\eqref{eq:limit}, we obtain the Minkowski sum
\[
\partial_T (Q|G) = \overline{T_\IF \ln (\relint V_G)} + T_\JF \ln (\relint G_J) .
\]
By the induction hypothesis (ii) for $V_G$, 
\[
\dim \overline{T_\IF \ln (\relint V_G)} \le \dim V_G
\]
and hence
\begin{align*}
\dim \partial_T (Q|G) & \le \dim \overline{T_\IF \ln (\relint V_G)} + \dim \left( T_\JF \ln (\relint G_J) \right) \\
& \le \dim V_G + \dim G \\
& \le \codim_Q G - 1 + \dim G \\
&= \dim Q - 1 .
\end{align*}

\item[$\bullet$] 
$b \neq 0$:

After passage to a subsequence,
any sequence $v^k \in \relint V_G$ giving rise to a finite value
must satisfy
\[
v^k \to v^* \in V_G\setminus\relint(V_G) .
\]
Indeed, otherwise $T_\IF\ln v^k$ would remain bounded, whereas
$b \, \ln r^k$ diverges.

Let $\pi = \id_d - b \, b^\trans / (b^\trans b) \in \R^{d \times d}$
be the orthogonal projection onto $(\im b)^\perp$.
Applying $\pi$ to Eqn.~\eqref{eq:limit} eliminates the radial term, and 
\[
\pi (\partial_T (Q|G)) 
\subseteq
\left\{ 
\mu + \pi(T_\JF \ln q^*_\JF) \mid
\mu \in \partial_{\pi T_\IF} (V_G)
, \;
q^* \in \relint G
\right\} ,
\]
where the latter represents a Minkowski sum.

In this case, we only have inclusion: Existence of a sequence $v^k$ with ${\pi T_\IF \ln v^k \to \mu}$
does not imply existence of a sequence $(r^k,v^k)$ with $r^k \to 0$ such that
\[
b \, \ln r^k+T_\IF\ln v^k \to \gamma
\]
and $\mu=\pi \gamma$.

By the induction hypothesis (i) for $V_G$, 
\[
\dim \partial_{\pi T_\IF} (V_G) \le \dim V_G - 1
\]
and hence
\begin{align*}
\dim \pi (\partial_T (Q|G)) &\le \dim \partial_{\pi T_\IF} (V_G) + \dim G \\
& \le \dim V_G - 1 + \dim G \\
& \le \codim_Q G - 1 - 1 + \dim G \\
&= \dim Q - 2 .
\end{align*}
Since $\pi$ has one-dimensional kernel,
\begin{align*}
\dim \partial_T (Q|G) & \le \dim \pi (\partial_T (Q|G)) + 1 \\
& \le \dim Q - 2 + 1 \\
&= \dim Q - 1 .
\end{align*}
\end{itemize}

$-$
\quad
statement (ii):

Having established (i) for $Q$, we now prove (ii).
Since $Q$ is compact, every sequence in $\relint Q$ has a subsequence converging to a point of $Q$.
If the limit belongs to $\relint Q$, continuity of $T\ln$ yields a point in $T\ln(\relint Q)$.
If the limit belongs to $Q\setminus\relint Q$, the finite limit value
belongs to $\partial_T(Q)$.
Thus,
\[
\overline{T\ln(\relint Q)}
=
T\ln(\relint Q)\cup\partial_T(Q).
\]

Now, $T\ln$ is continuous and locally Lipschitz on $\relint Q$.
Exhausting $\relint Q$ by compact subsets yields
\[
\dim (T \ln (\relint Q)) \le \dim Q.
\]
By statement (i),
\[
\dim \partial_T(Q) \le \dim Q-1
\]
and hence
\[
\dim \overline{T\ln(\relint Q)}
\le
\dim Q.
\]
\end{proof}


\paragraph{Solutions vs.\ parameters.}
To show the closedness of $V_\mathrm{bd}$,
we consider solutions and parameters at the same time,
that is,
we introduce the set
\begin{equation}
\vv =
\left\{
(y,\gamma) \in P \times \R^d \mid
\ff(y) = \gamma
\right\} ,
\end{equation}
which is the graph of $\ff$,
\[
\vv
=
\left\{
(y,\ff(y)) \mid y \in P 
\right\} \subseteq P \times \R^d ,
\]
in particular, 
$
\dim \vv = \dim P .
$
We take the closure of $\vv$ in
$
\overline{P} \times \R^d,
$
and define the graph boundary
\begin{equation}
\bb
=
\overline{\vv} \setminus \vv .
\end{equation}


\begin{lemma}[Graph boundary] \label{lem:gb}
Let $V_\mathrm{bd}$ be the set of boundary parameter values and $\bb$ be the graph boundary.
Then,
\begin{enumerate}[(i)]
\item
$\bb = \overline{\vv} \cap (\partial P \times \R^d)$ and
\item
$V_\mathrm{bd} = \pi_\gamma (\bb)$.
\end{enumerate}
\end{lemma}
\begin{proof}
(i)
Let
\[
(y^k, \ff(y^k))
\to
(y, \gamma)
\]
with $y^k\in P$ and $y\in P$.
Since $\ff$ is continuous on $P$, 
\[
\ff(y^k) \to \ff(y) .
\]
Hence $\gamma = \ff(y)$ and
\[
(y, \gamma) \in \vv.
\]
That is,
\[
\overline{\vv} \setminus \vv \subseteq \partial P \times \R^d ,
\]
whereas
\[
\vv\subseteq P\times\R^d .
\]
Hence,
\[
\bb
=
\overline{\vv}\setminus\vv
=
\overline{\vv}\cap(\partial P\times\R^d).
\]

(ii)
Let $\gamma\in V_\mathrm{bd}$. Then there exist
\[
(y^k)_{k\in\N} \subseteq P
\quad \text{and} \quad
y^*\in\partial P
\]
such that
\[
(y^k,\ff(y^k))
\to
(y^*,\gamma) \in \overline{\vv} .
\]
By statement (i),
\[
(y^*,\gamma)\in\bb 
\]
and
\[
\gamma\in\pi_\gamma(\bb).
\]

Conversely, let
\[
\gamma\in\pi_\gamma(\bb).
\]
By statement (i), there exists $y^*\in\partial P$ such that
\[
(y^*,\gamma)\in\bb.
\]
Hence, there is a sequence
\[
(y^k)_{k\in\N} \subseteq P
\]
such that
\[
(y^k,\ff(y^k))
\to
(y^*,\gamma) ,
\]
which implies that $\gamma\in V_\mathrm{bd}$.
\end{proof}

%

\begin{lemma}[Boundary-value closedness] \label{lem:bv_closed}
The set $V_\mathrm{bd}$ of boundary parameter values is closed.
\end{lemma}
\begin{proof}
By Lemma~\ref{lem:gb} (i), 
the graph boundary $\bb = \overline{\vv} \cap (\partial P \times \R^d)$ is the intersection of two closed sets,
and hence $\bb$ is closed.

To show that $V_\mathrm{bd}$ is closed, let $(\gamma^k)_{k \in \mathbb{N}} \subseteq V_\mathrm{bd}$ be a sequence converging to $\gamma \in \mathbb{R}^d$. 
By Lemma~\ref{lem:gb}~(ii),
there is a corresponding sequence $(y^k)_{k \in \mathbb{N}} \subseteq \partial P$ such that $(y^k, \gamma^k) \in \bb$. 
Since $\partial P$ is compact,
$y^k$ converges to $y^* \in \partial P$, after passing to a subsequence and calling it $y^k$ again.
Hence, $(y^k, \gamma^k)$ converges to $(y^*, \gamma)$ in $\partial P \times \mathbb{R}^d$.
Since $\bb$ is closed, 
$(y^*, \gamma) \in \bb$ 
which implies $\gamma \in \pi_\gamma(\bb) = V_\mathrm{bd}$,
by Lemma~\ref{lem:gb} (ii).
\end{proof}

Now, we can conclude the analysis of boundary parameter values.
Let 
\begin{equation}
V_\mathrm{bd}^c
=
\R^d\setminus V_\mathrm{bd} .
\end{equation}

\begin{corollary}[Boundary-good parameters] \label{cor:boundary_good}
Let $\dim P \le d$.
Then, the set
$
V_\mathrm{bd}^c
$
is open and dense.
\end{corollary}
\begin{proof}
By Lemma~\ref{lem:bv_closed}, the set $V_\mathrm{bd}$ is closed.
By Lemma~\ref{lem:bv_dim},
$
\dim V_\mathrm{bd}
\le
\dim P-1 
\le
d - 1
< d
$
and hence $V_\mathrm{bd}$ has empty interior.
Thus, its complement $V_\mathrm{bd}^c$ is open and dense.
\end{proof}

\begin{lemma}[Properness away from boundary values] \label{lem:proper}
The restriction
\[
\ff \colon
\ff^{-1}(V_\mathrm{bd}^c)
\to
V_\mathrm{bd}^c
\]
is proper.
\end{lemma}
\begin{proof}
Let $K\subseteq V_\mathrm{bd}^c$ be compact.
We show that
\[
\ff^{-1}(K) \subseteq P
\]
is compact. 
Let
\[
(y^k)_{k\in\N}
\subseteq
\ff^{-1}(K).
\]
Since $\overline{P} \times K$ is compact, 
\[
(y^k, \ff(y^k))
\to
(y^*, \gamma)
\in
\overline{P} \times K
\]
after passing to a subsequence and calling it $y^k$ again.
If $y^*\in \partial P$, then
\[
(y^*,\gamma)\in\bb,
\]
and hence
\[
\gamma\in V_\mathrm{bd} ,
\]
contradicting $K\subseteq V_\mathrm{bd}^c$.
Hence, $y^*\in P$. By continuity (of $\ff$ on $P$),
\[
\ff(y^*)=\gamma\in K.
\]
Thus, $\ff^{-1}(K)$ is compact.
\end{proof}


\subsection{Critical parameter values} \label{sec:critical}

Finally, we define the set of critical parameter values (over the proper region).
Let
\begin{subequations}
\begin{equation}
\crit(\ff)
=
\left\{
y \in P \mid
\rank D\ff_y<d
\right\} ,
\end{equation}
and 
\begin{equation}
V_\mathrm{crit}
=
\ff \left(
\crit(\ff) \cap \ff^{-1}(V_\mathrm{bd}^c)
\right)
\subseteq V_\mathrm{bd}^c 
\end{equation}
\end{subequations}
be the set of critical parameter values over $V_\mathrm{bd}^c$.

\begin{lemma}[Critical values] \label{lem:crit}
The set $V_\mathrm{crit}$ is closed in $V_\mathrm{bd}^c$ and has Lebesgue measure zero.
\end{lemma}
\begin{proof}
The critical set $\crit(\ff)$ is closed in $P$.
Hence,
\[
\crit(\ff)\cap\ff^{-1}(V_\mathrm{bd}^c)
\]
is closed in $\ff^{-1}(V_\mathrm{bd}^c)$.
By Lemma~\ref{lem:proper}, the restriction
\[
\ff \colon \ff^{-1}(V_\mathrm{bd}^c)\to V_\mathrm{bd}^c
\]
is proper. Hence, it maps closed subsets to closed subsets.
Thus
\[
V_\mathrm{crit}
=
\ff \bigl(
\crit(\ff)\cap\ff^{-1}(V_\mathrm{bd}^c)
\bigr)
\]
is closed in $V_\mathrm{bd}^c$.

Since $\ff$ is smooth, Sard's theorem implies that its set of critical
values has Lebesgue measure zero. Hence $V_\mathrm{crit}$ has
Lebesgue measure zero.
\end{proof}


\subsection{Main result} \label{sec:main}

Now, we can combine our analyses of boundary and critical parameters.
Let
\begin{subequations}
\begin{equation}
V_\mathrm{bad}
=
V_\mathrm{bd} \cup V_\mathrm{crit}
\end{equation}
and
\begin{equation}
V_\mathrm{good}
=
\R^d \setminus V_\mathrm{bad}
=
\R^d \setminus V_\mathrm{bd} \setminus V_\mathrm{crit}
=
V_\mathrm{bd}^c \setminus V_\mathrm{crit}.
\end{equation}
\end{subequations}

\begin{theorem}[Generic finiteness] \label{thm:generic_finiteness}
Let $\dim P \le d$.
Then, 
\begin{enumerate}[(i)]
\item
$V_\mathrm{good}$ is open and dense in $\R^d$.
\item
For every $\gamma\in V_\mathrm{good}$, the fiber
$
\ff^{-1}(\gamma)
$
is finite (in particular, empty if ${\dim P < d}$).
\end{enumerate}
\end{theorem}

\begin{proof}
(i) By Corollary~\ref{cor:boundary_good},
$V_\mathrm{bd}^c$ is open, and by Lemma~\ref{lem:crit},
$V_\mathrm{crit}$ is closed in $V_\mathrm{bd}^c$.
Hence,
\[
V_\mathrm{good}
=
V_\mathrm{bd}^c \setminus V_\mathrm{crit}
\]
is open in $V_\mathrm{bd}^c$ and in $\R^d$.

Moreover, by Lemma~\ref{lem:bv_dim},
$V_\mathrm{bd}$ has Lebesgue measure zero, 
and by Lemma~\ref{lem:crit} again,
also $V_\mathrm{crit}$ has Lebesgue measure zero.
Thus,
\[
V_\mathrm{bad}
=
V_\mathrm{bd} \cup V_\mathrm{crit}
\]
has Lebesgue measure zero. In particular, its complement
$V_\mathrm{good}$ is dense in $\R^d$.

(ii) Let $\gamma \in V_\mathrm{good}$.
On the one hand,
since $\gamma \notin V_\mathrm{bd}$
and $\ff$ is proper on $\ff^{-1}(V_\mathrm{bd}^c)$ by Lemma~\ref{lem:proper},
the fiber $\ff^{-1}(\gamma)$ is compact.
On the other hand,
since $\gamma \notin V_\mathrm{crit}$ is a regular value of $\ff$,
the fiber $\ff^{-1}(\gamma)$ is a smooth submanifold of $P$ of dimension
\[
\dim P-d .
\]

If $\dim P=d$, then $\ff^{-1}(\gamma)$ is a compact, zero-dimensional manifold. 
Hence it is discrete and therefore finite.

If $\dim P<d$, then $\crit(\ff) = P$, since
\[
\rank D\ff_y\le\dim P<d ,
\quad
y\in P.
\]
Hence, $\ff(P) \cap V_{\mathrm{bd}}^c \subseteq V_\mathrm{crit}$.
Since $V_{\mathrm{good}} \subseteq V_{\mathrm{bd}}^c\setminus V_{\mathrm{crit}}$,
we have $V_\mathrm{good} \cap \ff(P) = \emptyset$,
that is,
\[
\ff^{-1}(\gamma)=\emptyset , 
\quad \gamma \in V_\mathrm{good} .
\]
\end{proof}

%

\paragraph{Remark.}

We note that the condition $\dim P \le d$ is sufficient for generic finiteness, but not necessary.
If $\dim P > d$,
then generic finiteness can occur only when the map is everywhere critical. 
In that case, the set of parameters with nonempty fiber has Lebesgue measure zero, 
and the set of parameters with empty fiber is dense; see Section~\ref{sec:under}.


\subsection{Restricted sets of (original) parameters} \label{sec:main_original}

Recall that we are interested in the generic finiteness of
the solution set $Y_c$ for positive parameters
\[
c\in\R^m_{>}.
\]
In applications, we often restrict to a positive parameter set
\[
\mathcal C\subseteq\R^m_{>}.
\]
Moreover, we sometimes distinguish between variable parameters $c_i>0$,
$i\in I_\mathrm{var}$, and fixed parameters $c_i=1$,
$i\in I_\mathrm{fix}$, where
$
I_\mathrm{var} \, \dot\cup \, I_\mathrm{fix}
=
\{1,\ldots,m\}.
$
Then,
\[
\mathcal C
=
\mathcal C_\mathrm{var}
\times
\mathcal C_\mathrm{fix} ,
\]
where
$
\mathcal C_\mathrm{var}
\subseteq
\R_{>}^{I_\mathrm{var}}
$
is a set of variable parameters and
$
\mathcal C_\mathrm{fix}
=
\{1\}^{I_\mathrm{fix}} .
$
We assume that
\[
\emptyset\neq
\mathcal C_\mathrm{var}
\subseteq
\overline{\int_(\mathcal C_\mathrm{var})},
\]
where the interior and closure are taken in
$\R^{I_\mathrm{var}}_>$.
In particular, $\mathcal C_\mathrm{var}$ has full dimension
$|I_\mathrm{var}|$.
This condition includes open or closed full-dimensional boxes,
indeed, full-dimensional semialgebraic regions (without lower-dimensional components).

We define the logarithmic lumping map
\begin{equation}
\Lambda:
\R^m_{>}
\to
\R^d,
\qquad
c
\mapsto
H^\trans\ln c.
\end{equation}
Then
\[
Y_c
=
\left\{
y\in P
\ \middle|\
y^H=c^H
\right\}
=
\left\{
y\in P
\ \middle|\
\ff(y)=\Lambda(c)
\right\}.
\]
That is,
\[
Y_c
=
\mathcal{Y}_\gamma
=
\left\{
y \in P
\; \middle| \;
\ff(y) = \gamma
\right\}
\]
with $\gamma = \Lambda(c)$.
If
$
\mathcal C
=
\mathcal C_\mathrm{var}
\times
\mathcal C_\mathrm{fix} 
$
and $c = (c_\mathrm{var}, c_\mathrm{fix}) \in \mathcal C$,
then
\[
\Lambda(c)
=
H_\mathrm{var}^\trans\ln c_\mathrm{var},
\]
where $H_\mathrm{var}$ is the submatrix of $H$ consisting of the rows
indexed by $I_\mathrm{var}$.
Accordingly, we define
\[
\Lambda_\mathrm{var} \colon
\R_{>}^{I_\mathrm{var}}
\to
\R^d,
\qquad
c_\mathrm{var}
\mapsto
H_\mathrm{var}^\trans\ln c_\mathrm{var}.
\]

The logarithm and exponential maps are diffeomorphisms and locally
bi-Lipschitz. Hence they preserve Lebesgue measure-zero sets.
If
\[
\rank H_\mathrm{var}=d,
\]
then $\Lambda_\mathrm{var}$ is a submersion. In particular, the
preimage of a Lebesgue measure-zero subset of $\R^d$ under
$\Lambda_\mathrm{var}$ has again Lebesgue measure zero.

Now, we can formulate a version of Theorem~\ref{thm:generic_finiteness} for restricted, original parameters.

\begin{theorem}[Generic finiteness for restricted, original parameters] \label{thm:generic_finiteness_original}
For the parametrized system of generalized polynomial equations $A \, (c \circ x^B) = 0$,
let $\dim P \le d$ .

Further, let $\mathcal C = \mathcal C_\mathrm{var} \times \mathcal{C}_\mathrm{fix}$,
where $\emptyset \neq \mathcal C_\mathrm{var} \subseteq \overline{\int_(\mathcal{C}_\mathrm{var})} \subseteq \R_>^{I_\mathrm{var}}$ 
and $\rank H_\mathrm{var} = d$.
Finally, let
\[
\mathcal C_\mathrm{good}
=
\left\{
c \in \mathcal C \; \middle| \;
H_\mathrm{var}^\trans \ln c_\mathrm{var} \in V_\mathrm{good}
\right\}.
\]
Then, 
\begin{enumerate}[(i)]
\item
$\mathcal C_\mathrm{good}$ is 
open and dense in $\mathcal C$.
\item
For every $c\in\mathcal C_\mathrm{good}$, the solution set $Y_c$ is finite (in particular, empty if $\dim P < d$).
\end{enumerate}
\end{theorem}

\subsection{Application to parametrized power-law dynamical systems} \label{sec:app}

We apply Theorem~\ref{thm:generic_finiteness_original} to parametrized power-law dynamical systems (with {\em real} rather than integer exponents),
for example, arising from reaction networks with (generalized) mass-action kinetics.
We also allow for linear conservation relations. 

Consider the dynamical system for the nonnegative (concentration) vector $x \in \R^n_\ge$,
\[
\dd{x}{t} = N \left( k \circ x^V \right) ,
\]
where $N \in \R^{n \times m}$ is the (stoichiometric) coefficient matrix, $V \in \R^{n \times m}$ is a {\em real} exponent matrix, and $k \in \R^m_>$ is the vector of positive rate constants. 

We are interested in positive steady states, in particular, in positive solutions $x \in \R^n_>$ to 
\begin{equation}
N \left( k \circ x^V \right) = \zeros
\end{equation}
for given positive rate constants $k \in \R^m_>$.

$\bullet$ \quad
If $\rank N = n$, then $N \, ( k \circ x^V \, ) = \zeros$ is already of the form
\[
A \left( c \circ x^B \right) = \zeros
\]
with $A=N$, $B=V$, and $c=k$.
Clearly, $\rank A = n$ and $L = B \, (\ker \cayley) \subseteq \R^n$, for any decomposition into classes. 
Hence, $\dim L \le \rank A$, that is,
\[
\dim P \le d .
\]
The positive parameter set $\mathcal{C} = \R^m_>$ is unrestricted,
all rate constants are variable, 
and a basis matrix $H$ of the monomial-dependency subspace has rank $d$. 
Hence,
Theorem~\ref{thm:generic_finiteness_original} applies:
If $\dim P = d$, then $L=\R^n$, and generically, there are finitely many positive solutions.
If $\dim P < d$, then generically, there are no solutions.

$\bullet$ \quad
If $r = \rank N < n$, the dynamical system has $\cons = n - r$ linear conservation relations, defining forward invariant sets.
In this case, we are interested in positive steady states in a given forward invariant set.

For the moment, we assume that the left kernel of $N$ has a {\em nonnegative} basis matrix $W \in \R^{\cons \times n}_\ge$,
that is, $\im W^\trans = \ker N^\trans$. 
Then, $\dd{}{t} (W x) = \zeros$, that is,
\[
W x = T 
\]
with nonnegative total amounts $T \in \R^\cons_\ge$ that are constant over time.
In other words, every forward-invariant set is defined by a vector of total amounts.

Altogether, we are interested in positive solutions $x \in \R^n_>$ to 
\begin{equation} \label{eq:ss}
N \left( k \circ x^V \right) = \zeros, \quad W x = T
\end{equation}
for positive parameters $k \in \R^m_>$ and necessarily also {\em positive} parameters $T \in \R^\cons_>$.

To bring system \eqref{eq:ss} into the form
\[
A \left( c \circ x^B \right) = \zeros ,
\]
we first rewrite the conservation relations as
\[
\id_\cons \left( T \circ x^{\zeros_{n \times \cons} }\right) 
- W \left( \ones_n \circ x^{\id_n} \right)
= \zeros .
\]
Then, we define a combined coefficient matrix $A \in \R^{(n + \cons) \times (m + \cons + n)}$,
a combined exponent matrix $B \in \R^{n \times (m + \cons + n)}$, 
and a combined parameter vector $c \in \R^{m + \cons + n}_>$,
\[
A = \begin{pmatrix} N & 0 & 0 \\ 0 & \id_\cons & -W \end{pmatrix} , \quad
B = \begin{pmatrix} V & 0_{n \times \cons} & \id_n \end{pmatrix} , \quad
c = \begin{pmatrix} k \\ T \\ \ones_n \end{pmatrix} .
\]
In particular, the parameter vector~$c$ contains {\em constant} parameters $1_n$ arising from the linear term $x = \ones_n \circ x^{\id_n}$ in the conservation relations.

Clearly, $\rank A = r + \cons = n$ and $L = B \, (\ker \cayley) \subseteq \R^n$, for any decomposition into classes. 
Hence, $\dim L \le \rank A$, that is, $\dim P \le d$.
If we choose the trivial decomposition
or the obvious decomposition into $\ell = 2$ classes (steady states and conservation relations), 
then $\dim L = n$ and hence
\[
\dim P = d .
\]
For $\ell=2$, the ``problem dimension'' equals $\dim P = d = m + \cons - 2$, that is, number of variable parameters (rate constants and total amounts) minus two,
which can be larger than~$n$, the dimension of the dynamical system.

It remains to guarantee that the logarithmic lumped parameter space, generated by the {\em variable} parameters, equals $\R^d$.
Let
\[
H = \begin{pmatrix} H_k \\ H_T \\ H_\ones \end{pmatrix} \in \R^{(m+\cons+n) \times d} .
\]
Indeed, we show that
\[
\rank H_\mathrm{var} = \rank \begin{pmatrix} H_k \\ H_T \end{pmatrix} = d .
\]

For $\ell=2$,
we construct a particular basis matrix $H$ of the monomial-dependency subspace $D = \ker B \cap \ker \cayley$.
Recall
\[
B = \begin{pmatrix} V & 0_{n \times \cons} & \id_n \end{pmatrix}
\quad \text{and} \quad
\cayley = \begin{pmatrix} 
\ones^\trans_m & 0 & 0 \\ 
0 & \ones^\trans_\cons & \ones^\trans_n
\end{pmatrix}
,
\]
where the Cayley matrix $\cayley \in \R^{2 \times (m + \cons + n)}$ indicates the $2$ classes.
Explicitly, we choose the following structure:
\[
H = 
\bordermatrix{& \gray{d} \cr
\gray{m} & H_k \cr 
\gray{\cons} & H_T \cr 
\gray{n} & H_\ones
}
=
\bordermatrix{& \gray{m-1} & \gray{\cons-1} \cr
& I_m & 0 \cr
& U & I_\cons \cr
& V' & 0
} ,
\]
that is, a $3 \times 2$ block matrix
with crucial blocks $I_m \in \R^{m \times (m-1)}$ and $I_\cons \in \R^{\cons \times (\cons-1)}$.
In particular, we set
\[
I_i = \begin{pmatrix}
\id_{i-1} \\ - \ones_{i-1}^\trans
\end{pmatrix} .
\]
Then, $V' = -V I_m$ to fulfill $B H = 0$.
To fulfill $\cayley H = 0$,
we have to choose $U$ such that $\ones^\trans_\cons U + \ones^\trans_n V' = 0$,
for example, $U = u_i \ones^\trans_n V I_m$ with any unit vector $u_i \in \R^\cons$.

Obviously, 
\[
\rank \begin{pmatrix} H_k \\ H_T \end{pmatrix}
= m-1 + \cons-1
= d .
\]
Again Theorem~\ref{thm:generic_finiteness_original} applies:
Generically, there are finitely many positive solutions.

\paragraph{Remark \rm(phantom monomials).}
The nonnegative matrix $W$ may contain zero columns. 
In this case, the corresponding monomials (individual variables $x_i$) do not appear in the equation $W x = T$. 
Still, the general framework in~\cite{MuellerRegensburger2023a} allows for this situation.

\paragraph{Remark \rm(nonpositive total amounts).}
The matrix $W$ of linear conservation relations need not be nonnegative.
In this case, the corresponding total amounts $T_i$ may have any sign.
Fix a sign vector $\sigma\in\{\pm1\}^\cons$
and consider the corresponding parameter orthant
\[
\mathcal O_\sigma
=
\left\{
T\in\R^\cons
\mid
\sigma_iT_i>0
\text{ for all }i
\right\}.
\]
Then, the equation $Wx=T$ can be written as
\[
W' x := \diag(\sigma) W x = \diag(\sigma) T =: T' > 0 ,
\]
and the genericity result applies on each parameter orthant. 
Since the coordinate hyperplanes 
\[ 
\bigcup_{i=1}^\cons \{T \in \R^\cons \mid T_i=0\} 
\] 
have Lebesgue measure zero and are closed, 
the final genericity statement holds globally in the space of total amounts.

\bigskip

We summarize the application of Theorem~\ref{thm:generic_finiteness_original} to parametrized power-law dynamical systems.
\begin{corollary}
A parametrized power-law dynamical system $\dd{x}{t} = N \, ( c \circ x^V )$ is given by a (stoichiometric) coefficient matrix $N \in \R^{n \times m}$ and 
a {\em real} exponent matrix $V \in \R^{n \times m}$.
Linear conservation relations $Wx=T$ are given by a basis matrix $W \in \R^{\cons \times n}$,
where $\cons = n - \rank N$.

For generic positive rate constants $k \in \R^m_>$ and generic total amounts $T \in \R^\cons$, 
there are finitely many positive solutions $x \in \R^n_>$ to
\[
N \left( k \circ x^V \right) = \zeros, \quad W x = T .
\]
That is, generically a parametrized power-law dynamical system (with {\em real} exponents) has finitely many positive steady states in a forward-invariant set (``stoichiometric compatibility class'').
\end{corollary}


\subsection{The case $\dim P > d$} \label{sec:under}

Generic finiteness may hold for $\dim P > d$, but (if and) only if the system is ``everywhere critical''.
We consider this property for general maps.

\begin{lemma}[Everywhere critical maps] \label{lem:critical_everywhere}
Let $\ff \colon P \to \R^d$ be smooth. The following statements are equivalent:
\begin{enumerate}[(i)]
\item $\crit(\ff)=P$ (the map is everywhere critical).
\item $\ff(P)$ has Lebesgue measure zero in $\R^d$.
\item $\ff(P)$ has empty interior in $\R^d$.
\item $\R^d\setminus\ff(P) = \left\{ \gamma \in \R^d \mid \ff^{-1}(\gamma) = \emptyset \right\}$
(the set of parameters with empty fiber)
is dense in~$\R^d$.
\end{enumerate}
\end{lemma}
\begin{proof}
(i) $\implies$ (ii): 
Let $\crit(\ff)=P$. Then $\ff(P) = \ff(\crit(\ff))$, 
and Sard's theorem implies that $\ff(\crit(\ff))$,
the set of critical values, has Lebesgue measure zero.

(ii) $\implies$ (iii): 
A subset of $\R^d$ with Lebesgue measure zero cannot contain any open ball, and thus must have empty interior.

($\neg$i) $\implies$ ($\neg$iii): 
Suppose there exists $y_0\in P$ with $\rank D \ff_{y_0} = d$.
By the submersion theorem, the image $\ff(U)$ of an open neighborhood $U$ of $y_0$
is open in $\R^d$,
and hence $\ff(P)$ does not have empty interior in $\R^d$.

(iii) $\iff$ (iv): 
A subset of $\R^d$ has empty interior if and only if its complement is dense.
\end{proof}

Now, we consider the case $\dim P > d$ (the ``dimension excess'').

\begin{proposition}[Dimension excess dichotomy] \label{prop:dim_excess}
Let $\dim P>d$.
Then exactly one of the two alternatives holds:
\begin{enumerate}[(i)]
\item
$\crit(\ff)=P$.

In this case,
$\ff(P)$ has Lebesgue measure zero, 
and the set of parameters $\gamma$ with empty fiber $\ff^{-1}(\gamma)$
is dense in~$\R^d$.

\item
There exists $y_0\in P$ with $\rank D\ff_{y_0}=d$.

In this case, 
there exists an open neighborhood $V$ of $\ff(y_0)$ such that,
for every $\gamma \in V$, the fiber
$
\ff^{-1}(\gamma)
$
contains a smooth submanifold of dimension
$
\dim P-d>0.
$
In particular, every such fiber is {\em infinite}.
\end{enumerate}
\end{proposition}
\begin{proof}
(i) If $\crit(\ff)=P$, the conclusions follow from Lemma~\ref{lem:critical_everywhere}.

(ii) Otherwise ($\crit(\ff) \neq P$), there exists $y_0\in P$ with $\rank D\ff_{y_0}=d$.
By the submersion theorem, $\ff$ maps an open neighborhood $U \subseteq P$ of $y_0$ to an open neighborhood $V := \ff(U) \subset \R^d$ of $\ff(y_0)$. 
Hence, for every parameter $\gamma \in V$, the regular value theorem (applied locally) dictates that the intersection $\ff^{-1}(\gamma) \cap U$ is a smooth submanifold of $U$ with dimension $\dim P - d > 0$. 
Hence, the entire fiber $\ff^{-1}(\gamma)$ is infinite.
\end{proof}


\section{Parametrized generalized polynomial inequalities} \label{sec:inequ}

Finally, we consider parametrized systems of generalized polynomial {\em inequalities}, as introduced in~\cite{MuellerRegensburger2023a},
thereby covering generalized semi-algebraic sets. 

Let $C \subseteq \R^m_>$ be a positive, blunt coefficient cone, $B \in \R^{n \times m}$ be a {\em real} exponent matrix, 
and $c \in \R^m_>$ a positive parameter vector.
We consider the inclusion
\begin{equation}
\left( c \circ x^B \right) \in C
\end{equation}
for the positive variable vector $x \in \R^n_>$.

As a necessary condition for the existence of solutions, we assume $C \neq \emptyset$.
We further assume that $C = C^1 \times \cdots \times C^\ell$,
that is, the coefficient cone has a direct product decomposition, inducing $\ell$ classes of (monomial) indices.
We can always choose the trivial decomposition with $\ell=1$.
To maximally reduce the problem dimension, we can choose the finest decomposition.

\paragraph{Crucial technical objects from Section~\ref{sec:equ}.}
The coefficient set $P = C \cap \Delta$ arises from the intersection with a product $\Delta = \Delta^{m_1-1} \times \cdots \times \Delta^{m_\ell-1}$ of (standard) simplices with $m_1+\ldots+m_\ell=m$.
If $C$ is given by linear equations and (non-)strict inequalities, 
then it is a polyhedral cone (with some faces removed by strict inequalities).

Besides the coefficient set (arising from the coefficient cone $C$),
the monomial-dependency subspace (arising from the exponent matrix $B$) is the most crucial geometric object.
It is given by $D = \ker B \cap \ker \cayley$, where $\cayley \in \{0,1\}^{\ell \times m}$ is a Cayley matrix (having $\ell \times \ell$ blocks, in particular, rows of ones in the diagonal blocks and zeros elsewhere) indicating classes of indices.

\paragraph{Main previous result.}
By \cite[Theorem~5]{MuellerRegensburger2023a}, the solution set 
\begin{equation}
X_c = \{ x \in\R^n_> \mid \left( c \circ x^B \right) \in C \}
\end{equation}
is given by
\[
X_c = (Y_c / c)^E \circ \e^{L^\perp} ,
\]
where
\begin{equation}
Y_c =  \{ y \in P \mid y^H = c^H \}
\end{equation}
is the solution set on the coefficient set, involving a basis matrix $H \in \R^{m \times d}$ of the monomial-dependency subspace $D \subset \R^m$,
that is, $D = \im H$ and $d = \dim D$.
Note that solvability depends on only $d<m$ lumped parameters $c^H$.

\paragraph{Technical objects continued.}
$E \in \R^{n \times m}$ is an exponentiation matrix (explicitly mapping elements of~$Y_c$ to elements of~$X_c$), see~\cite{MuellerRegensburger2023a},
and $L = B \, (\ker \cayley) \subseteq \R^n$ is the monomial difference subspace.
Both $E$ and $L$ do not depend on~$c$.
We note that 
\[
d = m - \ell - \dim L .
\]

If $L \neq \R^n$, then a nonempty solution set~$X_c$ necessarily contains exponential fibers of the form $x \circ \e^{L^\perp}$ and cannot be finite.
Hence, we consider the generic finiteness of the solution set~$Y_c$ (on the coefficient set), 
that is, the generic finiteness of exponential fibers in the (original) solution set~$X_c$.

Note that the dimension of $L$ depends on the decomposition of $C$ we choose. A finer decomposition leads to smaller $\ker \cayley$ and hence may lead to smaller $L = B \, (\ker \cayley)$ and hence larger $L^\perp$. 
That is, the maximal fibers may be immediately visible only for the finest decomposition.
In the trivial decomposition (or other decompositions between trivial and finest), they may be hidden in $Y_c$ which then cannot be finite.

\paragraph{Inequalities via equations.}
There are two natural ways to obtain a generic-finiteness criterion for a parametrized system of generalized polynomial inequalities
from the criterion for equation systems obtained in Section~\ref{sec:main_original}.

$\bullet$ \quad
First, one may consider only the equations defining the linear span of the coefficient cone,
but retain all monomials.
This yields the sufficient condition
\[
\dim P \le d
\]
for the relaxed equation system,
where $P$ is the coefficient set and $d$ is the dimension of the monomial-dependency subspace of the inequality system. 
See Section~\ref{sec:relaxation}.

$\bullet$ \quad
Second, one may consider only the monomials appearing in equations.
This yields the sufficient condition
\[
\dim P^\mathrm{red} \le d^\mathrm{red}
\]
for the reduced equation system.
See Section~\ref{sec:reduction}.

Since every solution of the inequality system is a solution of the relaxed and reduced equation systems, 
generic finiteness for the equation systems implies generic finiteness for the inequality system.

Neither of the two sufficient conditions implies the other in general.
Indeed, removing inequality-only monomials can change both the dimension of the coefficient set and the dependency.

\subsection{Linear-span relaxation} \label{sec:relaxation}

We retain all monomials from the inequality system and hence remain in the ambient space $\R^m_>$ of the positive coefficient cone $C$.
Let
\[
S = \lin C , \quad 
C^\mathrm{rel} = S \cap \R^m_> , \quad \text{and} \quad
P^\mathrm{rel} = C^\mathrm{rel} \cap \Delta .
\]
That is,
$P^\mathrm{rel}$ is the coefficient set obtained by replacing the inequality constraints by the equations defining the linear subspace~$S$.
By construction,
\[
P^\mathrm{rel} \supseteq P , \quad
\aff P^\mathrm{rel} = \aff P , \quad \text{and} \quad
\dim P^\mathrm{rel} = \dim P .
\]

Since the exponent matrix $B$ and the Cayley matrix $\cayley$ are unchanged under this relaxation,
the monomial-dependency subspace
$
D = \ker B \cap \ker \cayley
$
and its dimension
$
d = \dim D 
$
are unchanged as well.

Hence, the generic-finiteness criterion $\dim P^\mathrm{rel} \le d$ (that is, $\dim P \le d$) for equation systems applies to the linear-span relaxation. 
Since
\[
Y_c = \left\{ y \in P \; \middle| \; y^H = c^H \right\}
\subseteq
\left\{ y \in P^\mathrm{rel} \; \middle| \; y^H = c^H \right\} =: Y_c^\mathrm{rel} ,
\]
we can apply Theorem~\ref{thm:generic_finiteness_original}
to guarantee that $Y_c^\mathrm{rel}$ (and hence $Y_c$) is generically finite.

\paragraph{Technical objects.}
As in Section~\ref{sec:main_original}, 
we may restrict the positive parameters $c \in \R^m_>$ to a positive parameter set
$
\mathcal C\subseteq\R^m_{>}.
$
In particular, we allow for variable and fixed (set to one) parameters,
$
\mathcal C
=
\mathcal C_\mathrm{var}
\times
\mathcal C_\mathrm{fix} ,
$
where we assume
\[
\emptyset\neq
\mathcal C_\mathrm{var}
\subseteq
\overline{\int_(\mathcal C_\mathrm{var})}
\]
to ensure full-dimensionality of the set of logarithmic lumped parameters.
Accordingly, we assume $\rank H_\mathrm{var} = d$.

As in Section~\ref{sec:main}, 
we write $V_\mathrm{good}^\mathrm{rel}$ for the set of good logarithmic lumped parameters 
(that are not boundary values or critical values)
of the relaxed equation system, 
and introduce
\[
\mathcal C_\mathrm{good}
=
\left\{
c \in \mathcal C \; \middle| \;
H_\mathrm{var}^\trans \ln c_\mathrm{var} \in V_\mathrm{good}^\mathrm{rel}
\right\}.
\]

\begin{theorem}[Generic finiteness for inequalities via relaxation] \label{thm:generic_finiteness_relaxation}
Consider the parametrized system of generalized polynomial inequalities $(c \circ x^B) \in C$,
and let $\dim P \le d$.

Further, let $\mathcal C = \mathcal C_\mathrm{var} \times \mathcal{C}_\mathrm{fix}$,
where $\emptyset \neq \mathcal C_\mathrm{var} \subseteq \overline{\int_(\mathcal{C}_\mathrm{var})} \subseteq \R_>^{I_\mathrm{var}}$ 
and $\rank H_\mathrm{var} = d$.
Finally, let
$
\mathcal C_\mathrm{good}
=
\{
c \in \mathcal C \mid
H_\mathrm{var}^\trans \ln c_\mathrm{var} \in V_\mathrm{good}^\mathrm{rel}
\} .
$

Then, 
\begin{enumerate}[(i)]
\item
$\mathcal C_\mathrm{good}$ is 
open and dense in $\mathcal C$.
\item
For every $c\in\mathcal C_\mathrm{good}$, the solution set $Y_c$ is finite (in particular, empty if $\dim P < d$).
\end{enumerate}
\end{theorem}





\paragraph{Remark.}
If $\overline P \cap \partial \R^m_> = \emptyset$,
then no logarithmic divergence can occur along the boundary of~$P$. 
In this case, the elaborate boundary-value analysis performed in Section~\ref{sec:boundary} is not required at all,
and it suffices to deal with critical values. 
The generic finiteness conclusion still requires $\dim P\le d$.

\subsection{Dimension reduction} \label{sec:reduction}

The linear-span relaxation removes inequalities, but retains all monomials.
As above, let \[ S = \lin C \subseteq \R^m . \]
A monomial with index $ i\in \{1,\ldots,m\}$ is called {\em phantom} if it does not occur in the equation part. 
Equivalently,
\[
u_i \in S 
\]
for the unit vector $u_i \in \R^m$.
If there are phantom monomials,
they can be discarded when applying the generic finiteness criterion to the equation part.
After reordering, we only consider the remaining monomials with indices $I_\mathrm{red} = \{1,\ldots,m_\mathrm{red}\}$ and $m_\mathrm{red}<m$.
We project $S$ to
\[
S^\mathrm{red} \subseteq \R^{m_\mathrm{red}}
\]
and introduce the reduced coefficient cone and set,
\[
C^\mathrm{red} = S^\mathrm{red} \cap \R^{m_\mathrm{red}}_>  \quad \text{and} \quad
P^\mathrm{red} = C^\mathrm{red} \cap \Delta^\mathrm{red} .
\]
Here, we assume that $C^\mathrm{red} = C^1 \times \cdots \times C^{\ell_\mathrm{red}}$,
that is, the reduced coefficient cone has a direct product decomposition, inducing $\ell_\mathrm{red}$ classes of monomials.
Accordingly,
$\Delta^\mathrm{red} = \Delta^{m_1-1} \times \cdots \times \Delta^{m_{\ell_\mathrm{red}}-1}$ is a product of (standard) simplices with $m_1+\ldots+m_{\ell_\mathrm{red}}=m_\mathrm{red}$.

The exponent matrix $B^\mathrm{red} \in \R^{n \times m_\mathrm{red}}$ is obtained from $B$ by discarding columns defining phantom monomials,
and the monomial-dependency subspace is given by
$
D^\mathrm{red} = \ker B^\mathrm{red} \cap \ker J_{C^\mathrm{red}} ,
$
where $J_{C^\mathrm{red}} \in \{0,1\}^{\ell_\mathrm{red} \times m_\mathrm{red}}$ is the corresponding Cayley matrix.
The dependency
$
d_\mathrm{red} = \dim D^\mathrm{red} 
$
may differ from $d$, just as $\dim P^\mathrm{red}$ may differ from $\dim P$.

Still, the generic-finiteness criterion $\dim P^\mathrm{red} \le d_\mathrm{red}$ for equation systems applies to the dimension reduction.
By construction,
\[
X_c = (Y_c / c)^E \circ \e^{L^\perp} \subseteq (Y_c^\mathrm{red} / c^\mathrm{red})^{E^\mathrm{red}} \circ \e^{(L^\mathrm{red})^\perp} =: X_c^\mathrm{red}
\]
with
\[
Y_c^\mathrm{red} = \left\{ y \in P^\mathrm{red} \; \middle| \; y^{H^\mathrm{red}} = (c^\mathrm{red})^{H^\mathrm{red}} \right\}
\]
and $E^\mathrm{red}$, $L^\mathrm{red}$, and $H^\mathrm{red}$ defined accordingly.
If $L^\mathrm{red} = L$,
we can apply Theorem~\ref{thm:generic_finiteness_original}
to guarantee that $Y_c^\mathrm{red}$ (and hence $Y_c$) is generically finite.
\paragraph{Technical objects.}
As above,
we consider a given set of variable or fixed (set to one) positive parameters
$
\mathcal C =
\mathcal C_\mathrm{var} \times \mathcal C_\mathrm{fix}
\subseteq \R^{I_\mathrm{var} \cup I_\mathrm{fix}}_>
= \R^m_{>} ,
$
where we assume
$
\emptyset \neq \mathcal C_\mathrm{var}
\subseteq \overline{\int_(\mathcal C_\mathrm{var})} .
$
After discarding phantom monomials and reordering,
we project $\mathcal{C}$ to
\[
\mathcal{C}^\mathrm{red} = 
\mathcal{C}_\mathrm{var}^\mathrm{red} \times \mathcal{C}_\mathrm{fix}^\mathrm{red} 
\subseteq \R^{I_\mathrm{var}^\mathrm{red} \cup I_\mathrm{fix}^\mathrm{red}}_> 
= \R^{m_\mathrm{red}}_>
\]
with index sets $I_\mathrm{var}^\mathrm{red} = I_\mathrm{var} \cap I^\mathrm{red}$, 
$I_\mathrm{fix}^\mathrm{red} = I_\mathrm{fix} \cap I^\mathrm{red}$ 
and $I_\mathrm{var}^\mathrm{red} \,\dot\cup\, I_\mathrm{fix}^\mathrm{red} = I^\mathrm{red} = \{1,\ldots,m_\mathrm{red}\}$.
In particular,
\[
\emptyset \neq \mathcal C_\mathrm{var}^\mathrm{red}
\subseteq \overline{\int_(\mathcal C_\mathrm{var}^\mathrm{red})} .
\]
Accordingly, we assume $\rank H_\mathrm{var}^\mathrm{red} = d_\mathrm{red}$.

As above, 
we write $V_\mathrm{good}^\mathrm{red}$ for the set of good logarithmic lumped parameters of the reduced system
and introduce
\[
\mathcal C_\mathrm{good}
=
\left\{
c \in \mathcal C \; \middle| \;
(H_\mathrm{var}^\mathrm{red})^\trans \ln c_\mathrm{var}^\mathrm{red} \in V_\mathrm{good}^\mathrm{red}
\right\}.
\]

\begin{theorem}[Generic finiteness for inequalities via reduction] \label{thm:generic_finiteness_reduction}
For the parametrized system of generalized polynomial inequalities $(c \circ x^B) \in C$,
let $\dim P^\mathrm{red} \le d_\mathrm{red}$ and $L^\mathrm{red} = L$ hold for the reduced equation system,
obtained by discarding phantom monomials.

Further, let $\mathcal C = \mathcal C_\mathrm{var} \times \mathcal{C}_\mathrm{fix}$,
where $\emptyset \neq \mathcal C_\mathrm{var} \subseteq \overline{\int_(\mathcal{C}_\mathrm{var})} \subseteq \R_>^{I_\mathrm{var}}$ 
and $\rank H_\mathrm{var}^\mathrm{red} = d_\mathrm{red}$.
Finally, let
$
\mathcal C_\mathrm{good}
=
\{
c \in \mathcal C \mid
(H_\mathrm{var}^\mathrm{red})^\trans \ln c_\mathrm{var}^\mathrm{red} \in V_\mathrm{good}^\mathrm{red}
\} .
$

Then, 
\begin{enumerate}[(i)]
\item
$\mathcal C_\mathrm{good}$ is 
open and dense in $\mathcal C$.
\item
For every $c\in\mathcal C_\mathrm{good}$, the solution set $Y_c$ is finite (in particular, empty if $\dim P^\mathrm{red} < d_\mathrm{red}$).
\end{enumerate}
\end{theorem}




%


\subsection{Parametrized generalized semi-algebraic sets} \label{sec:semi}

We define a parametrized generalized semi-algebraic set 
by a finite Boolean combination of parametrized generalized polynomial equations and inequalities. 
After rewriting negated relations and passing to disjunctive normal form, 
every such set is a finite union of finite intersections of generalized polynomial equations and (strict or non-strict) inequalities. 
We show that it is a parametrized system of generalized polynomial inequalities in the sense of \cite{MuellerRegensburger2023a}.

We define a {\em basic} parametrized generalized semi-algebraic set
by a finite conjunction of (strict or nonstrict) inequalities and equations for positive variables $x \in \R^n_>$, 
\[ 
a^\trans \left( c \circ x^B \right) \ge 0 , \quad
a^\trans \left( c \circ x^B \right) > 0 , \quad \text{or} \quad
a^\trans \left( c \circ x^B \right) = 0 ,
\]
where $a \in \R^m$ is a real coefficient vector, 
$B \in \R^{n \times m}$ is a {\em real} exponent matrix, 
and $c \in \R^m_>$ is a positive parameter vector. 

We define a {\em general} parametrized generalized semi-algebraic set
as a finite disjunction of basic ones.
After collecting all monomial terms occurring in the Boolean combination, 
we may use one global real exponent matrix $B \in \R^{n\times m}$ and one global positive parameter vector $c \in \R^m_>$.

For the explicit definition, 
let $x \in \R^n_>$ be a positive variable vector, 
$B \in \R^{n \times m}$ be a {\em real} exponent matrix, 
and $c \in \R^m_>$ be a positive parameter vector.
Together, they globally define the positive vector of monomial terms $c \circ x^B$.

Now, let $\mathcal{J} = \{1, \dots, n_\ast \}$ be a finite index set representing disjunctions (of basic systems), 
and for each $j \in \mathcal{J}$, let $\mathcal{K}_j = \{1, \dots, n_j\}$ be a finite index set representing conjunctions (within a basic system). 
For $j \in \mathcal{J}$ and $k \in \mathcal{K}_j$, let $a_{jk} \in \R^m$ be a real coefficient vector
and $\bowtie_{jk} \in \{ \ge, >, = \}$ be a relation.

A parametrized generalized semi-algebraic set is given by
\begin{equation} \label{eq:Sc}
S_c = \bigcup_{j \in \mathcal{J}} \, \bigcap_{k \in \mathcal{K}_j} \left\{ x \in \R^n_> \mid a_{jk}^\trans \left( c \circ x^B \right) \bowtie_{jk} 0 \right\} .
\end{equation}

The corresponding positive, blunt coefficient cone is given by
\[
C = \bigcup_{j \in \mathcal{J}} \, \bigcap_{k \in \mathcal{K}_j} \left\{ y \in \R^m_> \mid a_{jk}^\trans \, y \bowtie_{jk} 0 \right\} .
\]
Every inequality or equation
$
a_{jk}^\trans y\bowtie_{jk}0
$
for
$
j\in\mathcal J
$
and
$
k\in\mathcal K_j,
$
as well as every positivity condition
$
y_i>0
$
for
$
i\in\{1,\ldots,m\},
$
defines a closed half-space, an open half-space, or a linear subspace
of \(\R^m\), according to
$
\bowtie_{jk}\in\{\ge,>,=\}.
$
Every intersection indexed by $j$ defines a partially open polyhedral cone. 
Their not necessarily disjoint union defines a positive, blunt cone that may have several nonconvex components.

Clearly, the parametrized generalized semi-algebraic set in~\eqref{eq:Sc} is given by
\[
S_c = \{ x \in \R^n_> \mid \left( c \circ x^B \right) \in C \} ,
\]
that is, a parametrized generalized semi-algebraic set (as defined above)
is a parametrized system of generalized polynomial inequalities as defined by M\"uller and Regensburger~\cite{MuellerRegensburger2023a}.


\subsubsection*{Acknowledgements}

This research was funded in whole, or in part, by the Austrian Science Fund (FWF), 
grant DOI 10.55776/PAT3748324 to SM.

%
%


\clearpage

\addcontentsline{toc}{section}{References}
\bibliographystyle{abbrv}

\bibliography{MR_overleaf,old,finite}

\begin{thebibliography}{10}

\bibitem{banajicraciun2010}
M.~Banaji and G.~Craciun.
\newblock Graph-theoretic criteria for injectivity and unique equilibria in
  general chemical reaction systems.
\newblock {\em Adv. in Appl. Math.}, 44(2):168--184, 2010.

\bibitem{banaji2026positive}
M.~Banaji and E.~Feliu.
\newblock Positive equilibria in mass action networks: geometry and bounds.
\newblock {\em Journal of Mathematical Biology}, 93(1):article no. 37, 2026.

\bibitem{BanajiPantea2016}
M.~Banaji and C.~Pantea.
\newblock Some results on injectivity and multistationarity in chemical
  reaction networks.
\newblock {\em SIAM J. Appl. Dyn. Syst.}, 15(2):807--869, 2016.

\bibitem{bernshtein1975}
D.~N. Bernshtein.
\newblock The number of roots of a system of equations.
\newblock {\em Functional Analysis and Its Applications}, 9(3):183--185, 1975.

\bibitem{bihan2007gale}
F.~Bihan and F.~Sottile.
\newblock New fewnomial upper bounds from {G}ale dual polynomial systems.
\newblock {\em Moscow Mathematical Journal}, 7(3):387--407, 2007.

\bibitem{coste2000}
M.~Coste.
\newblock {\em An Introduction to O-minimal Geometry}.
\newblock Lecture Notes. Universit{\`a} di {P}isa, Pisa, 2000.

\bibitem{craciunfeinberg2005i}
G.~Craciun and M.~Feinberg.
\newblock Multiple equilibria in complex chemical reaction networks. {I}. {T}he
  injectivity property.
\newblock {\em SIAM J. Appl. Math.}, 65(5):1526--1546, 2005.

\bibitem{Feinberg1972}
M.~Feinberg.
\newblock Complex balancing in general kinetic systems.
\newblock {\em Arch. Rational Mech. Anal.}, 49:187--194, 1972.

\bibitem{feliu2025generic}
E.~Feliu, O.~Henriksson, and B.~Pascual-Escudero.
\newblock Generic consistency and nondegeneracy of vertically parametrized
  systems.
\newblock {\em Journal of Algebra}, 677:630--666, 2025.

\bibitem{feliu2026generic}
E.~Feliu, O.~Henriksson, and B.~Pascual-Escudero.
\newblock The generic geometry of steady state varieties.
\newblock {\em SIAM Journal on Applied Algebra and Geometry}, 10(3):519--548,
  2026.

\bibitem{FeliuMuellerRegensburger2019}
E.~Feliu, S.~M{\"u}ller, and G.~Regensburger.
\newblock Characterizing injectivity of classes of maps via classes of
  matrices.
\newblock {\em Linear Algebra and its Applications}, 580:236--261, 2019.

\bibitem{feliuwiuf2013}
E.~Feliu and C.~Wiuf.
\newblock Injectivity, multiple zeros, and multistationarity in reaction
  networks.
\newblock {\em SIAM Journal on Applied Dynamical Systems}, 12(4):1685--1721,
  2013.

\bibitem{guilleminpollack1974}
V.~Guillemin and A.~Pollack.
\newblock {\em Differential Topology}.
\newblock Prentice-Hall, Englewood Cliffs, NJ, 1974.

\bibitem{Horn1972}
F.~Horn.
\newblock Necessary and sufficient conditions for complex balancing in chemical
  kinetics.
\newblock {\em Arch. Rational Mech. Anal.}, 49:172--186, 1972.

\bibitem{HornJackson1972}
F.~Horn and R.~Jackson.
\newblock General mass action kinetics.
\newblock {\em Arch. Rational Mech. Anal.}, 47:81--116, 1972.

\bibitem{khovanskii1991}
A.~G. Khovanskii.
\newblock {\em Fewnomials}, volume~88 of {\em Translations of Mathematical
  Monographs}.
\newblock American Mathematical Society, Providence, RI, 1991.

\bibitem{kushnirenko1976}
A.~G. Kushnirenko.
\newblock Newton polytopes and the {B}ezout theorem.
\newblock {\em Functional Analysis and Its Applications}, 10(3):233--235, 1976.

\bibitem{lee2013}
J.~M. Lee.
\newblock {\em Introduction to Smooth Manifolds}, volume 218 of {\em Graduate
  Texts in Mathematics}.
\newblock Springer, New York, 2 edition, 2013.

\bibitem{Mueller2016}
S.~M{\"u}ller, E.~Feliu, G.~Regensburger, C.~Conradi, A.~Shiu, and
  A.~Dickenstein.
\newblock Sign conditions for injectivity of generalized polynomial maps with
  applications to chemical reaction networks and real algebraic geometry.
\newblock {\em Found. Comput. Math.}, 16(1):69--97, 2016.

\bibitem{MuellerRegensburger2012}
S.~M{\"u}ller and G.~Regensburger.
\newblock Generalized mass action systems: {C}omplex balancing equilibria and
  sign vectors of the stoichiometric and kinetic-order subspaces.
\newblock {\em SIAM J. Appl. Math.}, 72(6):1926--1947, 2012.

\bibitem{MuellerRegensburger2014}
S.~M{\"u}ller and G.~Regensburger.
\newblock Generalized mass-action systems and positive solutions of polynomial
  equations with real and symbolic exponents.
\newblock In {\em Computer Algebra in Scientific Computing -- CASC 2014},
  volume 8660 of {\em Lecture Notes in Comput. Sci.}, pages 302--323,
  Berlin/Heidelberg, 2014. Springer.

\bibitem{MuellerRegensburger2016}
S.~M{\"u}ller and G.~Regensburger.
\newblock Elementary vectors and conformal sums in polyhedral geometry and
  their relevance for metabolic pathway analysis.
\newblock {\em Front. Genet.}, 7(90):1--11, 2016.

\bibitem{MuellerRegensburger2023a}
S.~M{\"u}ller and G.~Regensburger.
\newblock Parametrized systems of generalized polynomial inequalities via
  linear algebra and convex geometry.
\newblock {\em Positivity}, 30:4, 2026.

\bibitem{vandendries1998}
L.~van~den Dries.
\newblock {\em Tame Topology and O-minimal Structures}, volume 248 of {\em
  London Mathematical Society Lecture Note Series}.
\newblock Cambridge University Press, Cambridge, 1998.

\bibitem{vandendriesmiller1996}
L.~van~den Dries and C.~Miller.
\newblock Geometric categories and o-minimal structures.
\newblock {\em Duke Mathematical Journal}, 84(2):497--540, 1996.

\bibitem{wilkie1996}
A.~J. Wilkie.
\newblock Model completeness results for expansions of the ordered field of
  real numbers by restricted {P}faffian functions and the exponential function.
\newblock {\em Journal of the American Mathematical Society}, 9(4):1051--1094,
  1996.

\end{thebibliography}

\end{document}